\documentclass[a4paper,12pt]{article}
\usepackage{amsthm}
\usepackage{amsfonts}
\usepackage{amsmath}
\usepackage{amssymb}
\usepackage[pdftex]{graphicx}
\usepackage{hyperref}
\usepackage{xcolor}

\numberwithin{equation}{section}
\newtheorem{theorem}[equation]{Theorem}
\theoremstyle{definition}
\newtheorem{definition}[equation]{Definition}
\theoremstyle{remark}
\newtheorem{remark}[equation]{Note}
\theoremstyle{plain}
\newtheorem{lemma}[equation]{Lemma}
\newtheorem{corollary}[equation]{Corollary}
\newtheorem{proposition}[equation]{Proposition}

\newcommand{\diag}{\operatorname{diag}}

\begin{document}

\title{Curvature on Eschenburg spaces}
\author{Jason DeVito and Peyton Johnson}
\date{}

\maketitle

\begin{abstract}We investigate the curvature of Eschenburg spaces with respect to two different metrics, one constructed by Eschenburg and the other by Wilking.  With respect to the Eschenburg metric, we obtain a simple complete characterization of the curvature of every Eschenburg space in terms of the triples of integers defining the space.  With respect to Wilking's metric, we study all the examples whose natural isometry group acts with cohomogeneity two.  Here, we find that apart from the previously known examples with almost positive curvature, all the remaining examples have open sets of points with zero-curvature planes.
\end{abstract}

\section{Introduction}

Given two triples of integers, $p=(p_1,p_2,p_3)$ and $q=(q_1,q_2,q_3)$ with $\sum p_i = \sum q_i$, there is an action of the compact Lie group $S^1\subseteq \mathbb{C}$ on $SU(3)$, given by $$z \ast A = \diag(z^{p_1}, z^{p_2}, z^{p_3}) A \diag(z^{q_1}, z^{q_2}, z^{q_3})^{-1}.$$  When the action is free, the quotient space is a smooth manifold called an Eschenburg space and denoted by $E_{p,q}$.

Eschenburg spaces, which generalize the homogeneous Aloff-Wallach spaces \cite{AW}, were introduced by Eschenburg \cite{Es} where he showed an infinite sub-family of them admit Riemannian metrics of positive sectional curvature.  These provided the first inhomogeneous examples of positively curved Riemannian manifolds.

In more detail, Eschenburg equipped $SU(3)$ with a certain left $SU(3)$-invariant right $U(2)$-invariant Riemannian metric, endowing $SU(3)$ with a non-negatively curved Riemannian metric with fewer zero-curvature planes than in the standard bi-invariant metric.  The above $S^1$ action is isometric with respect to this new metric, and hence this metric descends to a metric on each Eschenburg space.  We will refer to this metric as the \textit{Eschenburg metric}.  We note that interchanging $p$ and $q$, as well as permuting the $q_i$ can change the isometry type, but not the diffeomorphism type of the Eschenburg space. Hence, the Eschenburg space will typically have six non-isometric Eschenburg metrics.

Using the notation $\underline{p} = \min\{p_1,p_2,p_3\}$ and $\overline{p} = \max\{p_1,p_2,p_3\}$, Eschenburg proved:

\begin{theorem}[Eschenburg]\label{thm:esc} The Eschenburg space $E_{p,q}$ has an Eschenburg metric of positive sectional curvature if and only if $q_i\notin [\underline{p},\overline{p}]$ for any $i=1,2,3$.
\end{theorem}

It is therefore natural to study the curvature properties of the remaining Eschenburg spaces.  This was begun in \cite{Ke1}, where Kerin found that $E_{p,q}$ for $p = (0,1,1)$ and $q = (0,0,2)$ is almost positively curved, but not positively curved.  Recall that a Riemannian manifold is said to be \textit{almost positively curved} if the set of points for which all $2$-planes are positively curved is open and dense.  Kerin also showed that with one exception, every Eschenburg space has at least one Eschenburg metric of quasi-positive curvature. We recall that a metric is called \textit{quasi-positively curved} if it is non-negatively curved and it has a point at which all $2$-planes are positively curved. 

Our first main result completely characterizes the nature of the curvature of every Eschenburg space with Eschenburg metric.  To describe our results, let $S_3$ denote the symmetric group on the set $\{1,2,3\}$.

\begin{theorem}\label{thm:main}  The curvature of $E_{p,q}$ with the Eschenburg metric is determined by the 6 products $(p_{\sigma(1)}-q_1)(p_{\sigma(2)}- q_2)$ for $\sigma \in S_3$ as follows:

\begin{enumerate}\item  If all 6 products are positive, $E_{p,q}$ is positively curved.

\item  If all $6$ products are non-negative, with at least one positive and one zero, $E_{p,q}$ is almost positively curved but not positively curved.

\item  If at least one product is positive and at least one is negative, then $E_{p,q}$ is quasi-positively curved but not almost positive curved.

\item  If all products are non-positive, then every point of $E_{p,q}$ has at least one zero-curvature plane.

\end{enumerate}

\end{theorem}

In case 1, the fact that all six products are positive implies that both $q_1,q_2< \underline{p}$ or both $q_1,q_2> \overline{p}$.  The fact that $\sum p_i = \sum q_i$ now implies that $q_3$ also does not lie in $[\underline{p},\overline{p}]$.  Hence, Theorem \ref{thm:esc} may be viewed as a special case of  Theorem \ref{thm:main}.  Like case 1, case 3 and 4 each comprise infinitely many examples.

On the other hand, in case 2, the hypothesis on the $6$ products implies that $(p_1,p_2,p_3)$ is a permutation of $(0,1,1)$ while $(q_1,q_2,q_3) = (0,0,2)$, see Proposition \ref{prop:KerinExample}.  Kerin \cite[Theorem 2.4]{Ke1} has already shown the Eschenburg metric is almost positively curved but not positively curved in this case.

\bigskip


The case of homogeneous and cohomogeneity-one Eschenburg metrics on Eschenburg spaces is completely understood.  All homogeneous Eschenburg spaces (i.e., Aloff-Wallach spaces) admit a homogeneous Eschenburg metric of positive sectional curvature, except for $E_{p,q}$ with $p = (0,0,0)$ and $q= (1,-1,0)$ \cite{AW}.  All cohomogeneity one Eschenburg spaces admit a cohomogeneity one Eschenburg metric of positive sectional curvature, except for $E_{p,q}$ with $p = (0,0,2)$ and $q = (1,1,0)$ \cite{Wu}, which admits a cohomogeneity one Eschenburg metric of almost positive curvature.

In \cite{Wi}, Wilking equipped $SU(3)$ with a different metric, which we will refer to as the \textit{Wilking metric} and showed that it induces almost positive curvature on the above exceptional homogeneous Eschenburg space.  This illustrates a general principle:  Wilking's metric construction tends to have fewer isometries but also fewer zero-curvature planes.  Thus, in the search for new examples with (almost) positive curvature, it is natural to look at cohomogeneity 2 Eschenburg spaces equipped with the Wilking metric.

It turns out (see Proposition \ref{prop:nochange}), that up to isometry, the examples whose ``natural'' isometry group (in the sense of \cite{GroveShankarZiller}) acts on $E_{p,q}$ with cohomogeneity at most $2$ are of the form $p = (0,0,q_1+q_2+q_3)$, $q = (q_1,q_2,q_3)$ with $q_i$ pairwise relatively prime, $q_1 + q_2 + q_3\geq 0$, and $q_1\geq q_2$.  Our next main theorem describes the curvature of Wilking's metric on many of these examples.

\begin{theorem}\label{thm:almpos}  Suppose $q_1, q_2,$ and $q_3$ are pairwise relatively prime integers with $q_1 + q_2 + q_3 > 0$, $q_1 > q_2$, $q_1 > 0$, and $q_2 q_3 <0$.  Set $p = (0,0,q_1 + q_2 + q_3)$ and $q = (q_1, q_2 , q_3)$.  Then, under any of the following three hypothesis, Wilking's metric on $E_{p,q}$ is not almost positively curved.

\begin{enumerate}\item $q_2 + q_3\geq 0$

\item  $q_2 < 0$ and $q_1+q_2 \geq 0$

\item  $q_3 < 0$ and $q_1+ q_3 > 0$

\end{enumerate}
\end{theorem}

While the numerous hypothesis on the $q_i$ in this theorem may seem arbitrary, it turns out these conditions correspond exactly to the complement of the known examples, see Proposition \ref{prop:choices}.

We now outline the rest of the paper.  Section \ref{sec:background} contains background information on Eschenburg spaces, including the construction of both the Eschenburg and Wilking metrics.  Section \ref{sec:ecurvature} contains the proof of Theorem \ref{thm:main}.  Section \ref{sec:wcurvature} contains the proof of Theorem \ref{thm:almpos}.  

\textbf{Acknowledgements}  Both authors were supported by NSF DMS 2105556.  They are grateful for the support.  In addition, both authors gratefully acknowledge support from the Bill and Roberta Blankenship Undergraduate Research endowment.  We would also like to thank two anonymous referees who comments have vastly improved and corrected the article.

\section{Background}\label{sec:background}

In this section, we cover the necessary background, beginning with definition of Eschenburg spaces and ending with a description of the two Riemannian metrics we will be considering.

\subsection{Eschenburg spaces}
In this section, we cover the construction of Eschenburg spaces.

Suppose $p = (p_1,p_2,p_3)$ and $q = (q_1,q_2,q_3)$ are each triples of integers with $\sum p_i = \sum q_i$. Then we obtain an action of the group $S^1\subseteq \mathbb{C}$ on $$SU(3) = \{A\in M_3(\mathbb{C}):A\overline{A}^t = I, \det(A) = 1\}$$ defined by $$z\ast A = \diag(z^{p_1}, z^{p_2}, z^{p_3}) A \diag(z^{q_1}, z^{q_2}, z^{q_3})^{-1}.$$

We will call a pair of triples $(p,q)$ \textit{admissible} if $$\sum p_i = \sum q_i \text{ and } \gcd(p_{\sigma(1)}-q_1, p_{\sigma(2)}-q_2) = 1 \text{ for all  }\sigma \in S_3,$$ the permutation group on $\{1,2,3\}$.  We observe that the relation $\sum p_i = \sum q_i$ implies that analogous $\gcd$s involving $q_3$ are automatically equal to $1$ if $(p,q)$ is admissible.  The importance of this definition is highlighted by the following proposition.  See, e.g., \cite[Proposition 21]{Es} for a proof.

\begin{proposition}  Consider the $S^1$ action on $SU(3)$ defined by $p = (p_1,p_2,p_3)$ and $q=(q_1,q_2,q_3)$ as above.  Then the action is free (and hence, the quotient is a manifold) if and only if $((p_1,p_2,p_3),(q_1,q_2,q_3))$ is admissible.
\end{proposition}

As mentioned in the introduction, Eschenburg spaces were introduced by Eschenburg \cite{Es}, where he showed an infinite family of them admit Riemannian metrics of positive sectional curvature.  

Some operations on $(p,q)$ do not alter the diffeomorphism type of the resulting Eschenburg spaces.  Once we have defined the Eschenburg and Wilking metrics, we will see that some of these operations do alter the isometry type while others do not.  We record some of these operations below.  For a triple $p = (p_1,p_2,p_3)$ and $m\in \mathbb{Z}$, the notation $p+m$ will refer to $(p_1+ m, p_2+m, p_3 + m)$.  For $\sigma \in S_3$, the symmetric group on $\{1,2,3\}$, the notation $p_\sigma$ will refer to $(p_{\sigma(1)}, p_{\sigma(2)}, p_{\sigma(3)}).$  Lastly, the notation $-p$ refers to $(-p_1,-p_2,-p_3)$.

\begin{proposition}\label{prop:nochange}Suppose $(p,q)$ is admissible. 
 Then the following modifications of $(p,q)$ result in diffeomorphic Eschenburg spaces.  \begin{enumerate} \item $(p',q') =(p+m,q+m)$ for $m\in \mathbb{Z}$

 \item  $(p',q') = (-p, -q)$

\item $(p', q') = (p_\sigma, q)$ where $\sigma(3) = 3$ 

\item $(p',q') = (p_\sigma,q)$ where $\sigma(3)\neq 3$

\item $(p',q') = (p,q_\sigma)$ where $\sigma(3) = 3$

\item $(p',q') = (p,q_\sigma)$ where $\sigma(3)\neq 3$

\item  $(p',q') = (q,p)$

\end{enumerate}
With respect to the Eschenburg metric, modifications 1,2,3,4, and 5 determine isometric Eschenburg spaces.  With respect to the Wilking metric, modifications 1,2,3,5, and 7 determine isometric Eschenburg spaces.

\end{proposition}

Since we have not yet precisely defined the Eschenburg and Wilking metrics, we will postpone the proof to Section \ref{sec:Wmetric}.

As mentioned in the introduction, apart from the positively curved examples, only two other Eschenburg spaces are known to admit almost positively curved Riemannian metrics: when $(p,q) = ((0,1,1), (0,0,2))$ and when $(p,q) = ((0,0,0), (1,0,-1))$.

\begin{definition} Suppose $(p,q)$ is admissible.  Then $(p,q)$ is \textit{exceptional} if it is related to either $((0,1,1),(0,0,2))$ or $((0,0,0),(1,0,-1))$ by a combination of the modifications in Proposition \ref{prop:nochange}.  \end{definition}

The case where some $q_i$ is equal to some $p_j$ will arise frequently, so it is prudent to collect some facts about this case.

\begin{proposition}\label{prop:equalpq}  Suppose $(p,q)$ is admissible with $p_1 = q_1$.  Then, up to the changes in Proposition \ref{prop:nochange}, $p$ has the form $(a,0,0)$ and $q$ has the form $(a,-1, 1)$ for some non-negative integer $a$. If, in addition, $p_1 = q_2$, then $(p,q)$ is exceptional.
\end{proposition}

\begin{proof}Since $(p,q)$ is admissible, for any $i,j\in \{2,3\}$, we have $$1 = \gcd(p_1-q_1, p_i - q_j) = |p_i - q_j|.$$  Since both $q_2,q_3$ differ from $p_2$ by $1$, either $q_2$ and $q_3$ differ by $2$, or $q_2 = q_3$.

Assume initially that $q_2$ and $q_3$ differ by $2$.  Since $p_i$ differs from both $q_2$ and $q_3$ by $1$, we must have $p_2 = p_3$.  Using Proposition \ref{prop:nochange}(1) with $t=-p_2$, $p$ now has the form $(a,0,0)$.  Since $q_2 = p_2 \pm 1$ and $q_3 = p_2 \mp 1$, $q = (a,-1,1)$ up to order.  Proposition \ref{prop:nochange}(2) then allows us to assume $a\geq 0$.

Thus, we may assume $q_2 = q_3$ and that both $p_2$ and $p_3$ differ from $q_2$ by $1$.  Again, via Proposition \ref{prop:nochange}(1), we may assume that $q_2 = 0$, so $p_2,p_3\in \{\pm 1\}$.    As $q_2 = 0$, $q$ now has the form $(a,0,0)$, where, as above, we may assume $a\geq 0$.  If $p_2 = p_3$, then we have the contradiction $q_1 = \sum q_i = \sum p_i = q_1 \pm 2.$  Thus $p_2\neq p_3$.  Since $p_2, p_3\in \{\pm 1\}$, we again obtain the form of the proposition.  This completes the proof of the first assertion.

For the second assertion, we assume that $p_1 = q_1 = q_2$.  Subtracting $p_1$ (Proposition \ref{prop:nochange}(1) ) from each entry of $p$ and $q$, we obtain the triples $p = (0,p_2-p_1, p_3-p_1)$ and $q = (0,0, q_3-p_1)$.  We assume $p_2 - p_1 \geq 0$ via Proposition \ref{prop:nochange}(2).  Admissibility now implies that $p_2 - p_1 = |p_3 - p_1| = 1$.  If $p_3 - p_1  = 1$, then the fact that $\sum p_i = \sum q_i$ gives $p = (0,1,1)$ and $q = (0,0,2)$.  On the other hand, if $p_3 - p_1 = -1$, then $p = (0,1,-1)$ and $q = (0,0,0)$.  Thus, in both cases, $(p,q)$ is exceptional, completing the proof.  

\end{proof}

\subsection{The Eschenburg metric}\label{sec:Emetric}
In this section, we outline the construction of the metric.

Suppose $G$ is a compact Lie group with bi-invariant metric $\langle \cdot , \cdot \rangle_0$.   Given any closed subgroup $K\subseteq G$, we perform a Cheeger deformation \cite{Ch1} of $\langle \cdot, \cdot \rangle_0$ in the direction of $K$.  Specifically, we can equip $G\times K$ with the metric $\langle \cdot, \cdot\rangle_0 + t\langle \cdot , \cdot \rangle_0|_{K}$ where $t\in (0,\infty)$ is some fixed parameter.  Then $K$ acts freely and isometrically via $k\ast( g_1, k_1) = (g_1 k^{-1}, k k_1)$, and thus one obtains a metric on the orbit space $G\times_K K$.  As bi-invariant metrics on Lie groups are well-known to have non-negative sectional curvature, O'Neill's formula \cite{On1} implies the metric on $G\times_K K$ is also non-negatively curved.

The map $(g_1,k_1)\mapsto g_1 k_1$ descends to a diffeomorphism $G\times_K K\cong G$.  Using the diffeomorphism to transfer the submersion metric on $G\times_K K$ to $G$, we obtain a new metric $\langle \cdot, \cdot \rangle_K$ on $G$, the Cheeger deformation of $\langle \cdot, \cdot\rangle_0$ in the direction of $K$.

The $G\times K$ action on itself given by $(g,k)\ast (g_1,k_1) = (gg_1, k_1 k^{-1})$ commutes with the previously described $K$ action, so induces an isometric $G\times K$ action on $(G, \langle \cdot, \cdot\rangle_K)$.  One can easily verify the corresponding $G\times K$ action on $G$ is nothing but $(g,k)\ast g_1 = g g_1 k_1^{-1}$.  In particular, $\langle \cdot, \cdot \rangle_K$ is left $G$-invariant and right $K$-invariant.

Since $\langle \cdot,\cdot\rangle_K$ is left invariant, it is defined by its value at the tangent space $T_e G$ to the identity element $e\in G$.  We identify $T_e G$ with the Lie algebra $\mathfrak{g}$ of $G$ in the usual fashion.  We use the notation $\mathfrak{k}\subseteq \mathfrak{g}$ to denote the Lie algebra of $\mathfrak{k}$, and we let $\mathfrak{p}$ denote the $\langle\cdot,\cdot\rangle_0$-orthogonal complement of $\mathfrak{k}$ in $\mathfrak{g}$.  Defining $\phi_K:\mathfrak{g}\rightarrow \mathfrak{g}$ by $\phi_K(X_\mathfrak{k} + X_\mathfrak{p}) = \frac{t}{t+1} X_{\mathfrak{k}} + X_\mathfrak{p}$, it is straightforward to verify that $\langle X,Y\rangle_K = \langle X, \phi_K(Y)\rangle_0$ for all $X,Y\in \mathfrak{g}$.  It is also straightforward to verify that the horizontal lift of a vector $\phi_K^{-1}(X)\in \mathfrak{g} = T_e G$ to $T_{(e,e)}(G\times K)\cong \mathfrak{g}\oplus \mathfrak{k}$ is given by $(X, \frac{1}{t} X_\mathfrak{k})$.

Using $Tr(X)$ to denote the trace of a square matrix $X$, we can now define the Eschenburg metric. 

\begin{definition}\label{def:esch}  The \textit{Eschenburg metric on $SU(3)$} is the metric obtained from the above construction with $G = SU(3)$, $K = U(2)$  embedded via $A\mapsto \diag(A,\overline{\det(A)})$, $t=1$, and $\langle X,Y\rangle_0 = -Tr(XY)$.

The \textit{Eschenburg metric on an Eschenburg space $E_{p,q}$} is the metric obtained as a submersion metric induced from the Eschenburg metric on $SU(3)$.

\end{definition}

\begin{remark}  In the above definition, we set $t=1$ for definiteness.  While we state Proposition \ref{prop:esch} and Theorem \ref{thm:kerineqns} below in terms of the Eschenburg metric as we have defined it, they are actually valid for any $t \in (0,\infty)$.  In particular, allowing a different choice of $t$ in Definition \ref{def:esch} has no effect on Theorems \ref{thm:main} and \ref{thm:almpos}.

\end{remark}

The $S^1$ actions on $SU(3)$ defined in Section \ref{sec:background} act via left and right multiplication by elements of $U(3)$, so are not obviously isometric with respect to the Eschenburg metric.  The next proposition indicates that they are nevertheless isometric.

\begin{proposition}\label{prop:stillisometric}  Suppose $(p,q)$ is admissible.  Then the $S^1$ action on $SU(3)$ given by $$ z\ast A = \diag(z^{p_1}, z^{p_2}, z^{p_3}) A \diag(z^{q_1}, z^{q_2}, z^{q_3})^{-1}$$ is isometric with respect to the Eschenburg metric.

\end{proposition}

\begin{proof}  Consider the action $$z\ast' A = \diag(z^{p_1'}, z^{p_2'}, z^{p_3'}) A \diag(z^{q_1'}, z^{q_2'}, z^{q_3'})^{-1},$$ where $p_i' = 3p_i - (p_1 + p_2 + p_3)$   and $q_i' = 3q_i - (q_1 + q_2 + q_3)$.  Since $\det(\diag(z^{p_1'}, z^{p_2'}, z^{p_3'})) = \det(\diag(z^{q_1'}, z^{q_2'}, z^{q_3'})) = 1$, this action is defined by a subset of $K\times K$, and hence is isometric.

On the other hand, since $\sum p_i = \sum q_i$, we find \begin{align*} z\ast' A &= \diag(z^{p_1'}, z^{p_2'}, z^{p_3'}) A \diag(z^{q_1'}, z^{q_2'}, z^{q_3'})^{-1}\\
&= \diag(z^{3p_1}, z^{3p_2}, z^{3p_3}) z^{-(p_1 + p_2 + p_3)} A z^{q_1 + q_2 + q_3} \diag(z^{3q_1}, z^{3q_2}, z^{3q_3})^{-1}\\
&= \diag(z^{3p_1}, z^{3p_2}, z^{3p_3}) A \diag(z^{3q_1}, z^{3q_2}, z^{3q_3})^{-1}.\end{align*}

In particular, the action $\ast$ is isometric when done at three times the speed.  But this clearly implies that $\ast$ itself is an isometric action.

\end{proof}

The curvature calculations we are interested in all occur on the Lie algebra level.  To that end, we recall that $\mathfrak{g} = \mathfrak{su}(3)$ consists of the $3\times 3$ skew-Hermitian complex matrices.  The subpaces $\mathfrak{k}$ and $\mathfrak{p}$ are given as follows:  $$\mathfrak{k} = \{ A = (a_{ij})\in \mathfrak{g}: a_{13} = a_{23} = a_{31} = a_{32} = 0\}$$ and $$\mathfrak{p} = \{A = (a_{ij}) \in \mathfrak{g}: a_{11} = a_{12} = a_{21} = a_{22} = a_{33} = 0\}.$$ 

Eschenburg \cite{Es} found a characterization of when a $2$-plane $\sigma \subseteq \mathfrak{g}$ has zero sectional curvature with respect to the Eschenburg metric.  To state his characterization, we set $Y_3 =  i\diag(1,1,-2)\in \mathfrak{k}$ and we set $Y_1 = i\diag(-2,1,1)\in\mathfrak{k}$.

\begin{proposition}[Eschenburg]\label{prop:esch}  Suppose $SU(3)$ is equipped with the Eschenburg metric.  Then, a $2$-plane $\sigma \subseteq \mathfrak{g}$ with zero-sectional curvature must contain at least one of the following two vectors:

1.  $Y_3 = i\diag(1,1,-2)$ or

2.  $Ad_k Y_1$ for some $k\in K=U(2)$.
\end{proposition}

From O'Neill's formula \cite{On1}, a zero-curvature plane in $E_{p,q}$ with the Eschenburg metric must lift to a horizontal zero-curvature plane in $SU(3)$.  Using this fact and Proposition \ref{prop:esch}, Kerin \cite{Ke1} proved a characterization of points $[A]\in E_{p,q}$ having a zero-curvature plane with respect to the Eschenburg metric.  To describe his results, we set $P = \diag(p_1,p_2,p_3)$, $Q =\diag(q_1,q_2,q_3)$ and we view $Ad_k Y_1$ and $Y_3$ as left invariant vector fields on $SU(3)$.

\begin{theorem}[Kerin]\label{thm:kerineqns}  Suppose $E_{p,q}$ is equipped with the Eschenburg metric and let $A = (a_{ij})\in SU(3)$.  \begin{itemize} \item $Y_3$ is horizontal at $A$ if and only if $\langle Y_3, Ad_{A^{-1}}(iP)-iQ\rangle_0 = 0,$ which holds if and only if \begin{equation}\label{eqn:2.3} \sum_{j=1}^3 |a_{j3}|^2 p_j = q_3 \end{equation}

\item For $k = (k_{ij})\in K$, $Ad_k Y_1$ is horizontal at $A$ if and only if $\langle Ad_k Y_1, Ad_{A^{-1}}(iP) - iQ\rangle_0 = 0,$ which holds if and only if \begin{equation}\label{eqn:2.4} \sum_{j=1}^3 |((Ak)_{j1})|^2 p_j = |k_{11}|^2q_1 + |k_{21}|^2q_2\end{equation}

\end{itemize}
Moreover, if either $Y_3$ or $Ad_k Y_1$ is horizontal at $A$, then there is a zero-curvature plane at $[A]\in E_{p,q}$.

\end{theorem}

\begin{remark}Kerin's proof \cite[Theorem 2.2]{Ke1} of the last statement of Theorem \ref{thm:kerineqns} has a small gap as it relies on Proposition \ref{prop:esch}, which is incorrectly stated in \cite[Lemma 2.1]{Ke1} as a biconditional.  To complete the proof when $Ad_k Y_1$ is horizontal at $A$, one observes that the centralizer of $Y_1$ contains a two-dimensional subspace $\mathfrak{p}'\subseteq \mathfrak{p}$ consisting of matrices with $a_{13} = a_{31} = 0$.  In particular, after translating the seven-dimensional horizontal space at $A$ to $T_e G$, one finds it intersects $Ad_k \mathfrak{p}'\subseteq \mathfrak{p}$ in a subspace of dimension at least $1$.  Selecting a non-zero vector $Ad_k X$ from this subspace, one verifies that $[Ad_k Y_1, Ad_k X] = [Y_1, X] = 0$ and that $[(Ad_k Y_1)_{\mathfrak{k}}, (Ad_k X)_\mathfrak{k}] = 0$ since $(Ad_k X)_{\mathfrak{k}} = 0$.  From \cite[Example]{Es3}, it now follows that the plane $\operatorname{span}\{Ad_k Y_1, Ad_k X\}$ projects to a zero-curvature plane at $[A]$ in $E_{p,q}$.  The case where $Y_3$ is horizontal is similar but easier: here, the entire space $\mathfrak{p}$ centralizes $Y_1$, so one has a three dimensional set of $X$s to choose from.
\end{remark}

\begin{remark}\label{remark}  We will often find solutions to \eqref{eqn:2.4} as follows.  For $A\in SU(3)$, we consider the function $f_A:K\rightarrow \mathbb{R}$ given by the difference of the left hand side and right hand side of \eqref{eqn:2.4}.  The function $f_A$ is obviously continuous, and $K$ is a connected.  Thus, it is sufficient to show that $f_A$ takes on both non-positive and non-negative values.

\end{remark}

Recall the notation $\underline{p} = \min\{p_1,p_2,p_3\}$ and $\overline{p} = \max\{p_1,p_2,p_3\}$.  We will use the following estimate on the left hand side of \eqref{eqn:2.4}.

\begin{lemma}\label{lem:p}  For any $A\in SU(3)$ and any $k\in U(2)$, $$\underline{p}\leq \sum_{j=1}^3 |((Ak)_{j1})|^2 p_j \leq \overline{p}.$$

\end{lemma}

\begin{proof}  Notice that $(Ak)_{j1}$ comprises the first column of $Ak\in SU(3)$.  Thus $\sum_{j=1}^3 |(Ak)_{j1}|^2 = 1$ and obviously each term $|(Ak)_{j1}|^2$ is non-negative.  Thus, we have $$ \underline{p} = \sum_{j=1}^3 |(Ak)_{j1}|^2 \underline{p} \leq \sum_{j=1}^3 |(Ak)_{j1}|^2 p_j \leq \sum_{j=1}^3 |(Ak)_{j1}|^2\overline{p} =\overline{p}.$$


\end{proof}

We will need a more explicit form for the zero curvature planes containing $Ad_k Y_1$.  Set $$\mathcal{Z} = \left\{\begin{bmatrix} \beta i & 0 & 0\\ 0 & \beta i & z\\ 0 & -\overline{z} & -2\beta i\end{bmatrix} \in \mathfrak{su}(3):\beta\in \mathbb{R}\text{ and } z\in \mathbb{C}\right\}.$$  We note that the elements of $\mathcal{Z}$ all commute with $Y_1$.  In fact, $\mathcal{Z}$ and $Y_1$ span the centralizer of $Y_1$.

\begin{lemma}\label{lem:othervec}  Suppose $\sigma\subseteq \mathfrak{g}$ is a zero-curvature plane with respect to the Eschenburg metric on $SU(3)$.  In addition, assume that $Ad_k Y_1\in\sigma$ for some $k\in K$.  Then $$\sigma = \operatorname{span}\{ Ad_k Y_1, Ad_k Z\}$$ for some $Z\in \mathcal{Z}$.

\end{lemma}

\begin{proof}
Because $Ad_k Y_1 \in \sigma$, we know that $\sigma$ is spanned by $Ad_k Y_1$ and some other vector $X$.  The plane $\sigma$ lifts to the horizontal plane $\hat{\sigma}$ in $G\times K$ which is spanned by the vectors $(\phi_K(Ad_k Y_1), \frac{1}{t} \phi_K(Ad_k Y_1)_\mathfrak{k})$ and $(\phi_K(X), \frac{1}{t}\phi_K(X)_\mathfrak{k})$.  Since $\sigma$ has zero-curvature, O'Neill's formula implies that $\hat{\sigma}$ must also have zero-curvature, which then implies that $[\phi_K(Ad_k Y_1), \phi_K(X)] = 0$.  But notice that $Y_1\in \mathfrak{k}$ so $Ad_k Y_1\in \mathfrak{k}$.  In particular, $\phi_K(Ad_k Y_1) = \frac{t}{t+1} Ad_k Y_1$, so we now conclude that $[Ad_k Y_1, \phi(X)] = 0$.   Writing $X = X_\mathfrak{k} + X_\mathfrak{p}$, we find $$ 0 = [Ad_k Y_1, \phi(X)] = \underbrace{\frac{t}{t+1}[Ad_k Y_1, X_\mathfrak{k}]}_{\in \mathfrak{k}} + \underbrace{[ Ad_k Y_1, X_\mathfrak{p}]}_{\in \mathfrak{p}}.$$  Thus, we conclude that $[Ad_k Y_1, X_\mathfrak{k}] = [Ad_k Y_1, X_\mathfrak{p}] = 0$, so $[Ad_k Y_1, X] = 0$.  Since $Ad_k$ is a Lie algebra isomorphism, $[Y_1, Ad_{k^{-1}} X]= 0$, so $Ad_{k^{-1}} X$ centralizes $Y_1$.

Thus, we have $Ad_{k^{-1}} X = \lambda Y_1 + Z$ for some $\lambda \in \mathbb{R}$ and $Z\in \mathcal{Z}$.  Thus, $X = \lambda Ad_k Y_1 + Ad_k Z$.  Since $Ad_k Y_1$ and $X$ are both in $\sigma$, so is $Ad_k Z$ and $\sigma = \operatorname{span}\{ Ad_k Z, Ad_k Y_1\}$.

\end{proof}






\subsection{The Wilking metric}\label{sec:Wmetric}

We now describe a general method, due to Wilking \cite{Wi}, for constructing non-negatively curved metrics on Lie group and their quotients.  When compared with Cheeger deformed metrics, these metrics typically have fewer zero-curvature planes.  We refer to this method as Wilking's doubling trick.

Consider two closed subgroups $K_1,K_2\subseteq G$, giving rise to two Cheeger deformations $\langle \cdot, \cdot \rangle_{K_i},$ with $i=1,2$. 
 Then $G\times G$ supports the metric $\langle \cdot, \cdot \rangle_{K_1} + \langle \cdot, \cdot \rangle_{K_2}$ which is left invariant and of non-negative sectional curvature.  The action of $G$ on $G\times G$ given by $g\ast(g_1,g_2) = (gg_1,gg_2)$ is free and isometric, so induces a Riemannian metric on the quotient space $G\backslash (G\times G)$.  This quotient space is diffeomorphic to $G$ (see, e.g., \cite{Es2}), with a diffeomorphism being induced from the map $(g_1,g_2)\mapsto (g_1^{-1} g_2)$. Thus this construction induces a new metric $\langle \cdot, \cdot\rangle_{K_1,K_2}$ on $G$.

The isometric $K_1\times K_2$ action on $G\times G$ obtained by right multiplication descends to an isometric action on $(G,\langle \cdot, \cdot \rangle_{K_1,K_2})$, which implies that this new metric is left $K_1$-invariant and right $K_2$-invariant.  In particular, if $L\subseteq K_1\times K_2$, then the induced action of $L$ on $G$ is isometric.

In the context of Eschenburg spaces, we take $L = S^1$, embedded into $U(3)\times U(3)$ as $z\mapsto (\diag(z^{p_1}, z^{p_2}, z^{p_3}), \diag(z^{q_1}, z^{q_2}, z^{q_3}))$.  Since $S^1$ is not a subset of $SU(3)\times SU(3)$, it is not immediately obvious that $G\backslash G\times G/ S^1$ is diffeomorphic to $E_{p,q}$.  Indeed, the notation $G\times G/S^1$ does not actually even make sense.  Nevertheless, because $\sum p_i = \sum q_i$, $S^1$ does act in a well-defined and isometric manor on $\Delta G\backslash (G\times G)$.  In more detail, we view $G\times G\subseteq U(3)\times U(3)$.  Then $L$ acts on $\Delta U(3)\backslash (U(3)\times U(3))$.  We observe that the natural map $\Delta G\backslash (G\times G)\rightarrow \Delta U(3)\backslash (U(3)\times U(3))$ is injective.  Indeed, if $(A,B), (C,D)\in G\times G$ are in the same $\Delta U(3)$ orbit, then we have $(UA,UB) = (C,D)$ for some $U\in U(3)$.  But then $\det(U) =\det(UA) = \det(C) = 1$, so that $(U,U)\in \Delta G$.  Viewing $\Delta G\backslash (G\times G)$ as  subset of $\Delta U(3)\backslash (U(3)\times U(3))$ and choosing $(A,B)\in G\times G$, the $S^1$ action takes $[(A,B)]$ to $z\ast [(A,B)]=[(A\diag(z^{p_1}, z^{p_2}, z^{p_3})^{-1}, B\diag(z^{q_1}, z^{q_2}, z^{q_3})^{-1})]$. 
 Choosing any matrix $U\in U(3)$ with $\det(U) = z^{p_1 + p_2 + p_3} = z^{q_1+q_2+q_3}$, we observe that $$[(A\diag(z^{p_1}, z^{p_2}, z^{p_3})^{-1}, B\diag(z^{q_1}, z^{q_2}, z^{q_3})^{-1})]$$ is equal to  $$[(UA\diag(z^{p_1}, z^{p_2}, z^{p_3})^{-1}, UB\diag(z^{q_1}, z^{q_2}, z^{q_3})^{-1})] \in \Delta G\backslash (G\times G).$$  In particular the $L$ action on $\Delta U(3)\backslash(U(3)\times U(3))$ preserves $\Delta G\backslash (G\times G)$.

 Moreover, the $S^1$ action on $\Delta G\backslash (G\times G)$ is isometric.  To see this, observe that via the diffeomorphism $\Delta G\backslash (G\times G) \cong G$, the $S^1$ action on $\Delta G\backslash(G\times G)$ becomes the action given in Proposition \ref{prop:stillisometric}.  In particular, the proof of Proposition \ref{prop:stillisometric} applies in this case as well.

\begin{definition}
  The \textit{Wilking metric on $SU(3)$} is the metric obtained by applying Wilking's doubling trick in the case where both copies of $G = SU(3)$ are equipped with the Eschenburg metric.  The \textit{Wilking metric on an Eschenburg space $E_{p,q}$} is the metric obtained as a submersion metric induced from the Wilking metric on $SU(3)$.

\end{definition}

Now that we have defined both the Eschenburg and Wilking metrics, we can prove Proposition \ref{prop:nochange}.

\begin{proof}(Proof of Proposition \ref{prop:nochange}).  We will work down the list of options for $(p',q')$ given in the statement of Proposition \ref{prop:nochange}.

1.  $(p',q') = (p+t,q+t)$.  In this case, the corresponding actions by $S^1$ are identical, so the two quotients are isometric with respect to both the Eschenburg and Wilking metrics.

2.  $(p',q') = (-p,-q)$.  In this case, the two actions differ only by a change of coordinates in the circle. Hence, the quotients are isometric with respect to both the Eschenburg and Wilking metrics.

3. and 4.  $(p',q') = (p_\sigma,q)$ for $\sigma \in S_3$.  In this case, let $R$ denote the matrix obtained by applying $\sigma\in S_3$ to the rows of the identity matrix $I$.  Observe that $\det R = \pm 1$.  Set $$S = \begin{cases} R& \text{ if } \det R = 1\\ R\diag(1,1,-1)& \text{ if } \det R = -1\end{cases}.$$

We observe that $S\in SU(3)$ and that $S$ conjugates $\diag(z^{p_1}, z^{p_2}, z^{p_3})$ to $\diag(z^{p_{\sigma(1)}}, z^{p_{\sigma(2)}}, z^{p_{\sigma(3)}})$, it follows that left multiplication by $S$, $L_S:SU(3)\rightarrow SU(3)$ induces a diffeomorphism between $E_{p,q}$ and $E_{p',q'}$.  As the Eschenburg metric is left $SU(3)$-invariant, $L_S$ is an isometry with respect to the Eschenburg metric.  On the other hand, the Wilking metric is only left $U(2$)-invariant, and $L_S\in U(2)$ if and only if $\sigma(3) = 3$.

5. and 6. $(p',q') = (p, q_\sigma)$ for $\sigma \in S_3$.  The proof is analogous to the previous case using right multiplication by $S^{-1}$ instead of left multiplication by $S$.  The isometry statement follows as in the previous case by noting that both the Eschenburg and Wilking metrics are only right $U(2)$-invariant.

7.  $(p',q') = (q,p)$.  In this case, the inverse map $\iota:SU(3)\rightarrow SU(3)$ induces a diffeomorphism between $E_{p,q}$ and $E_{p',q'}$.  To see the map $\iota$ induces an isometry with respect to the Wilking metric, note that viewing $SU(3)$ as $\Delta SU(3)\backslash (SU(3)\times SU(3))$, $\iota$ corresponds to the map interchanging the two factors of $SU(3)\times SU(3)$.   Because the same metric is used on both factors, interchanging them is an isometry.

\end{proof}

We now derive conditions under which a point $[A]\in E_{p,q}$ has a zero-curvature plane with respect to the Wilking metric.  We will abbreviate $\phi_{U(2)}$ to $\phi$, so $\phi(X) = \frac{t}{t+1} X_\mathfrak{k}+ X_{\mathfrak{p}}$.

Given $X\in \mathfrak{g}$, we set $$\widehat{X} = (-\phi^{-1}(Ad_{A^{-1}} X), \phi^{-1}(X)) \in \mathfrak{g}\oplus \mathfrak{g}$$. 
 We note under the $\Delta G$ action on $G\times G$, that every orbit contains a point of the form $(A,I)$.  Additionally, we recall that e.g., Kerin \cite[Section 1, Equation 9]{Ke2} has shown that under the Riemannian submersion $G\times G\rightarrow G\backslash (G\times G)/ S^1\cong E_{p,q}$, the horizontal space at a point $(A,I)$, after left translation to the identity, is $$\mathcal{H}_A = \{ \widehat{X}: X\in \mathfrak{g} \text{ and } \langle X, Ad_{A}(iP) -  iQ)\rangle_0=0\},$$ where $P = \diag(p_1,p_2,p_3)$ and $Q = \diag(q_1,q_2,q_3).$  We recall that we have chosen the bi-invariant metric on $SU(3)$ whose value at the identity is $\langle X,Y\rangle_0 = -Tr(XY)$.

\begin{proposition}\label{prop:eqns} Suppose $A = (a_{ij})\in SU(3)$ is a point for which there are elements $k = (k_{ij})\in U(2)$ and $0\neq Z\in \mathcal{Z}$ simultaneously satisfy all three of the following conditions:

\begin{equation}\label{eqn:xtrace}\tag{Condition A} \sum_{j=1}^3 |((Ak)_{j1})|^2 p_j = |k_{11}|^2q_1 + |k_{21}|^2q_2 \end{equation}

\begin{equation}\label{eqn:ytrace}\tag{Condition B} \langle Ad_{k}Z,Ad_{A^{-1}}(iP)-iQ\rangle_0 =0 \end{equation}

\begin{equation}\label{eqn:dep}\tag{Condition C}\left\{(Ad_{Ak} Y_1)_\mathfrak{p}, (Ad_{Ak} Z)_{\mathfrak{p}}\right\} \text{ are linearly dependent over }\mathbb{R}.\end{equation}

Then, with respect to the Wilking metric, there is at least one zero-curvature plane in $T_{[A]} E_{p,q}$.

\end{proposition}

\begin{proof}   Observe that under the Riemannian submersion $G\times G\rightarrow G\backslash (G\times G)/S^1\cong E_{p,q}$, the point $(A^{-1},e)$ maps to $[A]$.  We will construct a horizontal zero-curvature plane in $T_{(A^{-1},e)} G\times G$.  By a result of Tapp \cite{Ta2}, such a plane must  project to a zero-curvature plane in $T_{[A]} E_{p,q}$.

Consider first the vector $Ad_k Y_1 \in \mathfrak{g}$.  From Theorem \ref{thm:kerineqns}, \ref{eqn:xtrace} is equivalent to the condition that $\langle Ad_k Y_1, Ad_{A^{-1}}(iP) - iQ\rangle_0 = 0$.  In particular, $\widehat{Ad_k Y_1}\in \mathcal{H}_{A^{-1}}$.  Likewise, \ref{eqn:ytrace} implies that $\widehat{Ad_k Z}\in \mathcal{H}_{A^{-1}}$.  Thus, the plane $\hat{\sigma}  = \operatorname{span}\{\widehat{Ad_k Y_1}, \widehat{Ad_k Z}\}$ is, after left translating to $(A^{-1},e)$, a horizontal plane at $(A^{-1},e)$.  It remains to determine its curvature.

Because the metric on $G\times G$ is a product of Eschenburg metrics, clearly $\hat{\sigma}$ has zero-curvature if and only if its two projections $$\hat{\sigma_1} = \{-\phi^{-1}(Ad_{Ak} Y_1), -\phi^{-1}(Ad_{Ak} Z)\}$$ and  $$\hat{\sigma_2} = \{ \phi^{-1}(Ad_k Y_1), \phi^{-1}(Ad_k Z)\}$$ both have zero-curvature with respect to the Eschenburg metric.

We begin with the easier case of $\hat{\sigma_2}$.  Eschenburg \cite{Es3} has shown that with respect to the Eschenburg metric, $\hat{\sigma_2}$ has zero sectional curvature if and only if $[Ad_k Y_1, Ad_k Z] = 0$ and  $[(Ad_k Y_1)_{\mathfrak{k}}, (Ad_k Z)_{\mathfrak{k}}] = 0$.  For the first equation, we find $[Ad_k Y_1, Ad_k Z] = [Y_1 , Z ] = 0$.  For the second, we note that $(Ad_k Z)_{\mathfrak{k}} = \beta Y_3$.  In particular, it is in the centralizer of $\mathfrak{k}$, so $[(Ad_k Y_1)_\mathfrak{k}, (Ad_k Z)_\mathfrak{k}] = 0$.  Thus, $\hat{\sigma}_2$ has zero-curvature.

We finally turn attention to $\hat{\sigma}_1$.  As in the previous case, we need show that $[Ad_{Ak} Y_1, Ad_{Ak} Z] = 0$ and $[(Ad_{Ak} Y_1)_\mathfrak{k}, (Ad_{Ak} Z)_\mathfrak{k}] = 0$.  For the first, we have $$[Ad_{Ak}Y_1, Ad_{Ak} Z] = [Y_1, Z] = 0.$$

For the second condition, we argue as follows.  First observe that $[\mathfrak{k},\mathfrak{k}]\subseteq \mathfrak{k}$, $[\mathfrak{p},\mathfrak{p}] \subseteq \mathfrak{k}$, and $[\mathfrak{k},\mathfrak{p}]\subseteq \mathfrak{p}$.  Then, since $[Ad_{Ak} Y_1, Ad_{Ak} Z]= 0$, it follows that $$0= [Ad_{Ak} Y_1, Ad_{Ak} Z]_{\mathfrak{k}} = [(Ad_{Ak} Y_1))_\mathfrak{k}, (Ad_{Ak} Z))_\mathfrak{k}] + [(Ad_{Ak} Y_1))_\mathfrak{p}, (Ad_{Ak} Z))_\mathfrak{p}].$$  Noting that \ref{eqn:dep} implies $[(Ad_{Ak} Y_1))_\mathfrak{p}, (Ad_{Ak} Z))_\mathfrak{p}] = 0$, we conclude that $[(Ad_{Ak} Y_1))_\mathfrak{k}, (Ad_{Ak} Z))_\mathfrak{k}] = 0$ as well.  Thus, $\hat{\sigma}_1$ has zero-curvature as well.

\end{proof}

We now specialize to Eschenburg spaces $E_{p,q}$ with $p = (0,0,q_1+q_2+q_3)$ and $q = (q_1, q_2, q_3)$ where the $q_i$ are relatively prime integers.  We identify $S^1$ with $\{(\diag(1,1,z^{q_1 + q_2 + q_3}), \diag(z^{q_1}, z^{q_2}, z^{q_3})):z\in S^1\subseteq \mathbb{C}\} \subseteq U(3)\times U(3)$.  Then, the normalizer of $S^1\subseteq U(3)\times U(3)$ contains $U(2)\times T^2\subseteq U(3)\times U(3)$, where $U(2) = K\subseteq SU(3)$, and $T^2$ is the maximal torus of diagonal matrices in $SU(3)$.  Noting that $U(2)\times T^2\subseteq K\times K$, this implies that with respect to the Wilking metric on $E_{p,q} = SU(3)/S^1$, that left multiplication by $U(2)$ and right multiplication by $T^2$ are isometries.

\begin{proposition}\label{prop:F}  Suppose $E_{p,q}$ is equipped with a Wilking metric with $p = (0,0,q_1 + q_2 + q_3)$ and $q = (q_1,q_2,q_3)$.  Then the isometry group acts on $E_{p,q}$ with cohomogeneity at most two and every orbit passes through a point $[A]\in E_{p,q}$ where $A\in \mathcal{F}$ and where the subset $\mathcal{F}\subseteq SU(3)$ is defined as $$\mathcal{F} = \left\{ \begin{bmatrix} \cos(\alpha) & -\sin(\alpha) & 0\\ \cos(\theta)\sin(\alpha) & \cos(\theta) \cos(\alpha) & -\sin(\theta)\\ \sin(\theta)\sin(\alpha)  & \sin(\theta)\cos(\alpha) & \cos(\theta) \end{bmatrix}\in SU(3):\alpha,\theta\in[0,2\pi]\right\}.$$
\end{proposition}

\begin{proof}  The fact that the action is cohomogeneity at most two is well-known \cite[Section 3]{GroveShankarZiller}.   So, we need only show that every orbit passes through a point in the image of $\mathcal{F}$.

Recall that we have a $U(2)$-principal bundle $\pi:SU(3)\rightarrow \mathbb{C}P^2$ given by mapping $A = (a_{ij})\in SU(3)$ to its last row.  It follows that two elements in $SU(3)$ with the same image in $\mathbb{C}P^2$ are $K$-equivalent.  Since the $\{I\}\times T^2\subseteq U(2)\times T^2$ action commutes with the $U(2)\times \{(1,1)\}\subseteq U(2)\times T^2$ action, $\pi$ induces a $T^2$ action on $\mathbb{C}P^2$.  The proof will be complete if we can show that every $T^2$ orbit in $\mathbb{C}P^2$ intersects $\pi(\mathcal{F})$.  In fact, since clearly any unit length element of $\mathbb{R}^3$ has the form $(\sin(\theta)\sin(\alpha), \sin(\theta)\cos(\alpha), \cos(\theta))$, it is sufficient to show that every point in $\mathbb{C}P^2$ is $T^2$-equivalent to one with all coordinates real.

Parametrizing $T^2\subseteq SU(3)$ as $\diag(z,w,\overline{z}\overline{w})$, the induced $T^2$ action on $\mathbb{C}P^2$ is given by $$(z,w)\ast [z_1:z_2:z_3] = [\overline{z}z_1:  \overline{w}z_2: zwz_3].$$  If any $z_i = 0$, then we may obviously find a $T^2$-equivalent point whose coordinates are all real.  If all three $z_i$ are non-zero, choose $u$ with $u^3 = \frac{z_1 z_2 z_3}{|z_1 z_2 z_3|}$.  Let $z=\frac{\overline{u}z_1}{|z_1|}$ and $w=\frac{\overline{u}z_2}{|z_2|}$.

Then it is easy to verify that $(z,w)\ast[z_1:z_2:z_3] = (z,w)\ast[\overline{u} z_1: \overline{u}z_2: \overline{u} z_3]$ has all coordinates real, completing the proof.
\end{proof}

We have the following corollary.

\begin{corollary}\label{cor:sameorbit} Two matrices $A = (a_{ij}), B= (b_{ij}) \in SU(3)$ are in the same $U(2)\times T^2$ orbit if and only if $|a_{3i}| = |b_{3i}|$ for all $i\in \{1,2,3\}$.

\end{corollary}

\begin{proof}  Given the block form of $U(2)\subseteq SU(3)$, the action of $U(2)\times T^2$ on the bottom row of a matrix simply multiplies it by various unit length complex numbers.  In particular, if $A$ and $B$ are in the same orbit, that we must have $|a_{3i}| = |b_{3i}|$ for all $i\in \{1,2,3\}$.

For the converse, assume that $|a_{3i}| = |b_{3i}|$ for all $i\in \{1,2,3\}$.  From the proof of Proposition \ref{prop:F}, we know that $A$ and $B$ are orbit-equivalent if $\pi(A),\pi(B)\in \mathbb{C}P^2$ are equivalent under the $T^2$-action.  Again from the proof of Proposition \ref{prop:F}, we see that under the $T^2$-action, $\pi(A)$ is equivalent to a point of the form $[\sin(\theta_A)\sin(\alpha_A), \sin(\theta_A)\cos(\alpha_A),\cos(\theta_A)]\in \mathbb{C}P^2$ and that similarly $\pi(B)$ is $T^2$-equivalent to $[(\sin(\theta_B)\sin(\alpha_B), \sin(\theta_B)\cos(\alpha_B),\cos(\theta_B)]\in \mathbb{C}P^2.$  Observe that $|\cos(\theta_A)| = |a_{33}| = |b_{33}| = |\cos(\theta_B)|$ so that $\cos(\theta_A) = \pm \cos(\theta_B)$ .  In an analogous fashion, we see that $\sin(\theta_A)\sin(\alpha_A) = \pm \sin(\theta_B)\sin(\alpha_B)$ and that $\sin(\theta_A)\cos(\alpha_A) = \pm \sin(\theta_B)\cos(\alpha_B)$.

To get the signs to agree, we note that element $(-1,1)\in T^2$ transforms the point $[\sin(\theta_A)\sin(\alpha_A), \sin(\theta_A)\cos(\alpha_A), \cos(\theta_A)]$ to $[-\sin(\theta_A)\sin(\alpha_A), \sin(\theta_A)\cos(\alpha_A), -\cos(\theta_A)]$ which is the same point as $ [\sin(\theta_A)\sin(\alpha_A), -\sin(\theta_A)\cos(\alpha_A), \cos(\theta_A)].$  That is, $(-1,1)$ acts by changing the sign of the middle coordinate.  Similarly, one sees that $(1,-1)$ changes the sign of the first coordinate, while $(-1,-1)$ changes the sign of the last coordinate.  Clearly, by using using these elements, one may transform $[\sin(\theta_A)\sin(\alpha_A), \sin(\theta_A)\cos(\alpha_A), \cos(\theta_A)]$ to $[\sin(\theta_B)\sin(\alpha_B), \sin(\theta_B)\cos(\alpha_B), \cos(\theta_B)]$, completing the proof.

\end{proof}

We now isolate the choices of $(q_1,q_2,q_3)$ for which the corresponding Eschenburg spaces are not already know to have a metric of strictly or almost positive curvature.  We observe that via Proposition \ref{prop:nochange} we may assume that $q_1 + q_2 + q_3 \geq 0$ and that $q_1 \geq q_2$.

\begin{proposition}\label{prop:choices}  Suppose $q=(q_1,q_2,q_3)$ with $q_1, q_2$ and $q_3$ pairwise relatively prime integers for which $q_1 + q_2 + q_3 \geq 0$ and $q_1 \geq q_2$.  Let $p = (0,0,q_1+q_2+q_3)$.  If $E_{p,q}$ is not diffeomorphic to a known example with positive or almost positive sectional curvature, then all of the following must occur:

\begin{enumerate} \item[1a)] $q_1 + q_2 + q_3 > 0$

\item[1b)] $q_1  >0$

\item[1c)]  $q_1 > q_2$

\item[1d)] $q_2 q_3 <0$

 and at least one of the following must occur:

\item[2a)]  $q_2 + q_3\geq 0$

\item[2b)]  $q_2 < 0$ and $q_1 + q_2\geq 0$

\item[2c)]  $q_3 < 0$ and $q_1 + q_3 > 0$
\end{enumerate}

\end{proposition}

\begin{proof} As mentioned above, we assume that $q_1 + q_2 + q_3 \geq 0$ and $q_1\geq q_2$.  If $q_1 + q_2 + q_3 = 0$, then the resulting Eschenburg space is homogeneous.  Apart from Wilking's almost positively curved example \cite{Wi}, these homogeneous spaces admit homogeneous metric of strictly positive sectional curvature \cite{AW}.  Thus, we may assume $q_1 + q_2 + q_3 > 0$.  Further, if $q_1 = q_2$, the the resulting Eschenburg space admits a cohomogeneity one action.  In particular, it is either diffeomorphic to a positively curved example or to Kerin's almost positively curved example \cite{Ke1}.  Thus, we may assume $q_1 > q_2$.  In addition, if any $q_i = 0$, then the fact that the $q_i$ are relatively prime implies that $(q_1,q_2,q_3)$ is a permutation of $(1,1,0)$ or $(1,-1,0)$.  The first case gives an Eschenburg space diffeomorphic to Kerin's almost positively curved example via the diffeomorphism swapping $p$ and $q$, while the second case again gives Wilking's almost positively curved example.  Thus, we may assume that all three $q_i$ are non-zero.

If all three $q_i$ are positive, it is easy to see that $0\notin [\min\{q_i\}, \max\{q_i\}]$ and $q_1 + q_2 + q_3 \notin [\min\{q_i\}, \max\{q_i\}]$.  In particular, by Theorem \ref{thm:esc}, the resulting Eschenburg space is diffeomorphic to a positively curved example.

Thus, we may assume that at least one $q_i < 0$.  If $q_1 < 0$, the condition $q_1\geq q_2$ implies $q_2 < 0$ as well.  Then the condition $q_1 + q_2 + q_3 > 0$ implies that $q_3 > q_1+ q_2 + q_3 > 0$.  From this, it follows easily that $q_i\notin [0, q_1 + q_2 + q_3]$ for all $i$.  In particular, these Eschenburg spaces are known to admit strictly positively curved metrics by Theorem \ref{thm:esc}.   Thus, we may assume $q_1 > 0$.  If both $q_2$ and $q_3$ are negative, we now see that $q_i\notin [0,q_1+q_2+q_3]$ for all $i$, so, again from Theorem \ref{thm:esc}, the resulting Eschenburg space is positively curved.  Thus, we have $q_2 q_3 < 0$, completing the verification of $1a)$ through $1d)$.

We now assume that $2a)$ does not occur and show that least one of $2b)$ and $2c)$ must occur.  Because $q_2 + q_3 < 0$, $q_1\notin [0,q_1+q_2+q_3]$.  Thus, to have a new example, we must have $q_i\in[0,q_1 + q_2 + q_3]$ for at least one $i\in \{2,3\}$.  If this is true for $i=2$, we find $q_2 \geq 0$ (which, since $q_2\neq 0$, implies that $q_2 > 0$ and $q_3 < 0$), and $q_2 \leq q_1 + q_2 + q_3$.  In other words, we must have $q_3 < 0$ and $0 \leq q_1 + q_3$.  The case where $q_3\in [0, q_1 + q_2 + q_3]$ is analogous.

It remains to see that the case $q_3 < 0$ with $q_1 + q_3 = 0$ cannot occur.  To see this, observe that if $q_1 + q_3 = 0$, then since $q_1$ and $q_3$ are relatively prime, we must have $q_1 = 1$ and $q_3 = -1$.  Since $q_2$ and $q_3$ have opposite signs, we have the contradiction $1 = q_1> q_2 > 0$.

\end{proof}

\section{Determining the curvature of the Eschenburg metrics}\label{sec:ecurvature}

In this section, we prove Theorem \ref{thm:main}.  We divide up the proof depending on the nature of the six possible products $(p_{\sigma(1)}-q_1)(p_{\sigma(2)}-q_2)$ for $\sigma\in S_3$.

\subsection{\texorpdfstring{At least one $(p_{\sigma(1)}-q_1)(p_{\sigma(2)}-q_2) > 0$}{At least one of product > 0}}\label{sec:pos}.

In this section, we will prove the first three cases of Theorem \ref{thm:main}, when at least one $(p_{\sigma(1)}-q_1)(p_{\sigma(2)}-q_2) > 0$.  Because permuting the $p_i$ is an isometry (Proposition \ref{prop:nochange}), we may assume that $(p_1-q_1)(p_2-q_2) > 0$. We begin with \cite[Theorem 2.3]{Ke1} which asserts that in this case, there is at least one point of positive curvature.  

\begin{theorem}[Kerin]\label{thm:quasi}  Suppose $E_{p,q}$ is an Eschenburg space with the Eschenburg metric and that $(p_1-q_1)(p_2-q_2) > 0$.  Then for any diagonal matrix $A \in SU(3)$,  all $2$-planes in $T_{[A]} E_{p,q}$ have positive sectional curvature.

\end{theorem}

Of course, if all six products are positive, it is well-known that the resulting metric is positively curved \cite{Es}, so we will assume at least one product is non-positive.  We will break into two cases depending on whether there is a negative product or not.

\begin{proposition}\label{prop:KerinExample}  Suppose $(p_{\sigma(1)}-q_1)(p_{\sigma(2)}-q_2) \geq 0$ for all $\sigma\in S_3$.  Assume additionally that at least one product is positive and at least one product is zero.  Then, up to the isometric modifications in Proposition \ref{prop:nochange}, $p = (0,1,1)$ and $q = (0,0,2)$.

\end{proposition}

\begin{proof}  Because permuting the $p_i$ as as well permuting $q_1$ and $q_2$ are isometries, we may assume that $p_1 = q_1$.  From Proposition \ref{prop:equalpq}, we conclude that either $(p,q)$ has the form $((a,0,0), (a,\pm 1, \mp 1) )$ or that $(p,q)$ has the form $((a, 1,-1), (a,0,0))$, where $a$ is some non-negative integer.

We begin with the first form $(p,q) = ((a,0,0), (a, \pm 1, \mp 1))$.  If $a = 0$, then it is easy to see that all six products are non-positive.  In particular, this contradicts the assumption that at least one product is positive.  Hence, we may assume $a > 0$.

Consider now the product $$(p_1-q_2)(p_2-q_1) = (a - q_2)(-a) \geq 0.$$  
Since $a>0$, then $0<a\leq q_2\leq 1$,  this forces $a=1$ and $q_2=1$, and therefore $q=(1,1,-1)$. Using Proposition \ref{prop:nochange}(1) with $t = -1$ and then Proposition \ref{prop:nochange}(2), we find $(p,q) = ((0,1,1), (0,0,2))$.  This completes the analysis of the first form.

We now move on to the second form $(p,q) = ((a, 1,-1), (a,0,0))$.  As before, if $a= 0$ then it is easy to see that all products are non-positive, giving a contradiction.  Hence, $a > 0$.

Now, consider the product $$(p_3-q_1)(p_1-q_2) = (-1 - a)a.$$  Since $a > 0$, this product is negative, giving a contradiction.

\end{proof}                                                                                                                   

\begin{theorem}\label{thm:notalmost}Suppose that $(p_{\sigma(1)}-q_1)(p_{\sigma(2)}-q_2) < 0$ for some $\sigma \in S_3$.  Then the Eschenburg metric on $E_{p,q}$ contains an open set of points which have at least one zero-curvature plane.

\end{theorem}

\begin{proof}Since permuting the $p_i$ is an isometry (Proposition \ref{prop:nochange}), we may assume that $p_1 - q_1 < 0$ and $p_2-q_2 > 0$.  Now, consider the open set $$U=\left\{A = (a_{ij})\in SU(3): \sum_{i=1}^3 |a_{i1}|^2 p_i < q_1 \text{ and } \sum_{i=1}^3 |a_{i2}|^2 p_i > q_2 \right\}.$$  Notice that the identity $I\in U$, so $U\neq \emptyset$.

We claim that for every $A\in U$, that there is a zero-curvature plane at $T_{[A]} E_{p,q}$.  From Theorem \ref{thm:kerineqns}, it is sufficient to show that for each $A\in U$, there is a $k\in K$ solving   \eqref{eqn:2.4}.

Following Remark \ref{remark}, it is sufficient to show the function $f_A:K\rightarrow\mathbb{R}$ given by $$f_A(k) = \sum_{i=1}^3 |(Ak)_{j1}|^2 p_j - |k_{11}|^2 q_1 - |k_{21}|^2 q_2 $$ attains both non-negative and non-positive values.

To that end, observe that for the identity matrix $I$, that $$f_A(I) =  \sum_{i=1}^{3} |a_{i1}|^2 p_i - q_1,$$ which is negative since $A\in U$.

Similarly, selecting $k\in K$ with $(k_{11}, k_{21}) = (0,1)$
, we see that $$f_A(k) = \sum_{i=1}^3 |a_{i2}|^2 p_i - q_2,$$ which is positive since $A\in U$.



\end{proof}

With all of this in hand, we can now prove cases (1), (2), and (3) of Theorem \ref{thm:main}.

\begin{proof}(Proof of (1),(2), and (3) of Theorem \ref{thm:main})  If all six products are positive, the resulting Eschenburg spaces are all positively curved\cite{Es}.

On the other hand, if all six products are non-negative, with at least one positive and at least one zero, then Proposition \ref{prop:KerinExample} implies that up to isometry, $(p,q) = ((0,1,1),(0,0,2))$.  But Kerin \cite[Theorem 2.4]{Ke1} showed that for this example, the Eschenburg metric is almost positively curved but not positively curved.

Thus, we may assume some of the six products are positive while some are negative.  By Theorem \ref{thm:quasi}, the Eschenburg metric has quasi-positive curvature.  On the other hand, by Theorem \ref{thm:notalmost}, the Eschenburg metric is not almost positively curved.

\end{proof}

\subsection{All products non-positive}\label{sec:neg}

The goal of this section is to prove Theorem \ref{thm:main}(4).  We begin with a characterization of the $p$ and $q$ for which all six products are non-positive.  Recall the notation $\underline{p} = \min\{p_1,p_2,p_3\}$ and $\overline{p} = \max\{p_1,p_2,p_3\}$.

\begin{proposition}\label{prop:negcases}Suppose $(p_{\sigma(1)} - q_1)(p_{\sigma(2)}-q_2)\leq 0$ for all $\sigma\in S_3.$  Then one of the following three cases occur up to isometry:

\begin{enumerate}  \item Both $\min\{q_1,q_2\}\leq \underline{p}$ and $\max\{q_1,q_2\}\geq \overline{p}$.

 \item  $(p,q) = ((0,0,2), (0,1,1))$

 \item  $(p,q) = ((-1,0,1), (0,0,0))$

\end{enumerate}

\end{proposition}

\begin{proof}We will assume the first conclusion does not occur and show that either the second or third conclusion must hold.

Assuming the first conclusion does not hold, we conclude $\max\{q_1, q_2\} <  \overline{p}$ or $\min\{q_1,q_2\} > \underline{p}$.  If the first option occurs, we use Proposition \ref{prop:nochange}(2) to replace $q$ with $-q$ and $p$ with $-p$.  The condition $\max\{q_1,q_2\} < \overline{p}$ becomes $\min\{-q_1, -q_2\} > \min\{-p_1,-p_2,-p_3\}$.  In particular, we may assume without loss of generality that  $\min\{q_1,q_2\} > \underline{p}$.  Further, by using Proposition \ref{prop:nochange}(3)(4)(5), and (6), we may assume \begin{equation}\label{eqn:inequalities} \underline{p} = p_1\leq p_2\leq p_3 = \overline{p} \text{ and } p_1< q_1 \leq q_2.\end{equation}

We will now show that that condition that all six products $(p_{\sigma(1)}-q_1)(p_{\sigma(2)} - q_2)$ are non-positive together with \eqref{eqn:inequalities} implies that $(p,q)$ is exceptional.  From Proposition \ref{prop:equalpq}, to show $(p,q)$ is exceptional, it is sufficient to show that $q_1 = q_2 = p_i$ for some $i$, or that $p_2 = p_3 = q_i$ for some $i$.

If $q_2 > p_2$, then $(p_1 - q_1)(p_2-q_2) > 0$, giving a contradiction.  Thus, we have $p_1 < q_1 \leq q_2 \leq p_2 \leq p_3$.

Then both $(p_2-q_1)(p_3-q_2)\geq 0$ and $(p_3-q_1)(p_2-q_2)\geq 0$, so both products must be zero.  The case that $q_1 = p_3$ cannot occur for otherwise we see $q_1\leq p_2\leq p_3 = q_1$, so $q_2 = q_1 = p_2 = p_3$, which contradicts admissibility of $p$ and $q$.  Thus, $q_1\neq p_3$ which then implies that $q_2 = p_2$ and either $q_2 = p_3$ or $q_1 = p_2$.  In either case, $(p,q)$ must be exceptional.  We now focus on each exceptional case.

Taking into account the isometries in Proposition \ref{prop:nochange}, there are at most four Eschenburg metrics on the exceptional Eschenburg space $E_{p,q}$ with $(p,q) = ((0,0,2),(0,1,1))$.  They are represented by the pairs \begin{align*} ((0,0,2),(1,1,0)), &((-2,0,0), (-1,0,-1)), \\ &((-1,-1,0),(-2,0,0)), \text{ and } ((-1,-1,0), (0,0,-2)).\end{align*}

Of these, the first and last pairs have $(p_1 - q_1 )(p_2-q_2)> 0$ and the third satisfies the first conclusion of this proposition.  The second, up to isometry, the example given in the second conclusion of this proposition.

Similarly, up to isometry, there are at most three Eschenburg metrics on the exceptional Eschenburg space $E_{(p,q)}$ with $(p,q) = ((0,0,0),(1,-1,0))$, but both pairs with $p=(0,0,0)$ fall into the first conclusion of this proposition.  The remaining pair is $((-1,0,1), (0,0,0))$, giving the third conclusion of the proposition.

\end{proof}

We are now ready to prove Theorem \ref{thm:main}(4).

\begin{theorem}  Suppose $(p_{\sigma(1)} - q_1)(p_{\sigma(2)}-q_2)\leq 0$ for all $\sigma\in S_3.$   Then the Eschenburg metric on $E_{p,q}$ has a zero-curvature plane at every point.

\end{theorem}

\begin{proof}  One of case 1, 2, or 3 of Proposition \ref{prop:negcases} must occur.  
The strategy for each case is the same as in the proof of Theorem \ref{thm:notalmost}.  Namely, we will show the function $f_A$ from Remark \ref{remark} attains both non-negative and non-positive values.

\textbf{Case 1:}  Assume that case 1 occurs.  Recall that by Lemma \ref{lem:p}, that for all $A\in SU(3)$ and $k\in U(2)$, that $\underline{p}\leq \sum_{j=1}^3 |(Ak)_{j1}|^2 p_j \leq \overline{p}.$

Using the identity $I\in K$, we find $$f_A(I) = p_1|a_{11}|^2+p_2|a_{21}|^2+p_3|a_{31}|^2-q_1.$$ Since $q_1\leq \underline{p}$, Lemma \ref{lem:p} implies $f_A(I)\geq 0$.

On the other hand, if we select $k\in K$ with $(k_{11},k_{21}) = (0,1)$, then $$f_A(k) = p_1|a_{12}|^2+p_2|a_{22}|^2+p_3|a_{32}|^2-q_2.$$ Since  $q_2\geq \overline{p}$, Lemma \ref{lem:p} implies $f_A(k)\leq 0$. This concludes the first case.




\bigskip

\textbf{Case 2:}  We now move onto the second case.  Here, the function $f_A$ takes the form \begin{equation*}f_A(k) = 2|a_{31} k_{11} + a_{32} k_{21}|^2 - |k_{21}|^2. \end{equation*}


If $a_{31} = a_{32} = 0$, then obviously $f_A(k)\leq 0$.  Otherwise, the matrix $$k = \frac{1}{\sqrt{|a_{31}|^2 + |a_{32}|^2}} \begin{bmatrix} \overline{a}_{32} & a_{31} &0 \\ - \overline{a}_{31} &  a_{32} & 0 \\ 0 & 0 & \sqrt{|a_{31}|^2 + |a_{32}|^2}  \end{bmatrix} \in K$$ and $f_A(k) = -|k_{21}|^2 \leq 0$.   On the other hand, for the identity matrix $I$, we have $f_A(I)\geq 0$.   This completes the second case.






\bigskip

\textbf{Case 3:} We lastly move on to the third case.  Here,  from our choice of $p$ and $q$, the function $f_A$ takes the form $$f_A(k) = -|a_{11}k_{11}+a_{12}k_{21}|^2+|a_{31}k_{11}+a_{32}k_{21}|^2.$$




As in the proof of Case 2, we can select $k\in U(2)$ to make $-|a_{11}k_{11}+a_{12}k_{21}|^2= 0$, so that $f_A(k)\geq 0$. We can also select $k\in U(2)$ to make $|a_{31}k_{11}+a_{32}k_{21}|^2=0$, so that $f_A(k)\leq 0$ as well. This completes the proof of the third case, and hence, of the theorem.
\end{proof}








\section{Wilking metrics with many zero-curvature planes}\label{sec:wcurvature}

In this section, we prove Theorem \ref{thm:almpos}.  We begin by setting up notation.  We let $p = (0,0,p_3)$ and $q = (q_1,q_2,q_3)$, where $p_3 = q_1+q_2+q_3$, and $q_1,q_2$, and $q_3$ are relatively prime integers.  Following Proposition \ref{prop:choices}, we assume that $p_3 > 0$ and that $q_1 > q_2$. We consider the Eschenburg space $E_{p,q}$ equipped with the Wilking metric.

We select a point $A = A(\theta,\alpha) \in \mathcal{F}$ so that $$A  = \begin{bmatrix} \cos(\alpha) & -\sin(\alpha) & 0 \\ \cos(\theta)\sin(\alpha) & \cos(\theta)\cos(\alpha) & -\sin(\alpha)\\ \sin(\theta)\sin(\alpha) & \sin(\theta)\cos(\alpha) & \cos(\theta) \end{bmatrix}.$$

We define $h(\theta)$ by the formula $$h(\theta) = \frac{(p_3-q_3)\cos^2(\theta) + q_3\sin^2(\theta)}{\sin^2(\theta)(p_3 \cos^2(\theta)+q_3)}$$ and we define $g(\theta)$ by the formula $$g(\theta) = \frac{p_3 q_1\cos^2(\theta)+p_3q_3\sin^2(\theta) -q_2q_3}{(q_1-q_2)(p_3 \cos^2(\theta)+q_3)}.$$

The importance of $h$ and $g$ is given by the following theorem which allows us to recognize when an Eschenburg space with Wilking metric is not almost positively curved.

\begin{theorem}\label{thm:recog}Suppose there is an real number $\theta_0$ for which both $0< h(\theta_0) < 1$ and $0 < g(\theta_0) < 1$.  Then there is a non-empty open subset of points in $E_{p,q}$ for which there is at least one zero-curvature plane at each point.

\end{theorem}

To prove Theorem \ref{thm:recog}, we require several lemmas, the first of which is a routine calculation.

\begin{lemma}\label{lem:identities} All of the following identities hold:

\begin{enumerate}\item $\displaystyle g(\theta) = \frac{p_3\sin^2(\theta)h(\theta)-q_2}{q_1-q_2}$

\item  $\displaystyle \sin^2(\theta)(1-h(\theta))= -\cos^2(\theta)\left(\frac{p_3\cos^2(\theta)-q_3}{p_3\cos^2(\theta)+q_3}\right)$

\item  $\displaystyle h'(\theta) = -\frac{2\sin(\theta)\cos(\theta)(p_3^2 \cos^4(\theta)+2p_3 q_3 \sin^2(\theta)\cos^2(\theta)-q_3^2)}{\sin^4(\theta)(p_3\cos^2(\theta)+q_3)^2}$

\item  $\displaystyle g'(\theta) = \frac{4p_3q_3^2 \sin(\theta)\cos(\theta)}{(q_1-q_2)(p_3\cos^2(\theta)+q_3)^2}$

\end{enumerate}
\end{lemma}

We continue with the next lemma.

\begin{lemma}\label{lem:important}  Suppose that  $\theta \in \mathbb{R}$ has all of the following properties:

\begin{itemize} \item Both $p_3 \cos^2(\theta) \pm q_3 \neq 0$

\item  $\theta$ is not an integral multiple of $\pi/2$

\end{itemize}

Additionally assume that there is an element $ k = (k_{ij})\in SU(2)\subseteq K$ and $\alpha \in \mathbb{R}$ for which both \begin{equation}\label{eqn:geqn} |k_{11}|^2 = g(\theta)\end{equation} and \begin{equation}\label{eqn:heqn} |\sin(\alpha)k_{11} - \cos(\alpha)\overline{k}_{12}|^2 = h(\theta) .\end{equation}  Then, with respect to the Wilking metric, there is a zero-curvature plane at $[A(\theta,\alpha)]\in E_{p,q}$.

\end{lemma}

\begin{proof}  Following Proposition \ref{prop:eqns}, we need to find $k\in K$, $0\neq Z\in \mathcal{Z}$ satisfying \ref{eqn:xtrace}, \ref{eqn:ytrace}, and \ref{eqn:dep}.  To that end, we use $k$ as hypothesized in the lemma.  We define $Z = \begin{bmatrix} \beta i & 0 & 0 \\ 0 & \beta i & z\\ 0 & -\overline{z} & -2\beta i\end{bmatrix}$ where $$z = 3i\tan(\theta)(\cos(\alpha)k_{11} + \sin(\alpha) \overline{k}_{12})$$ and where $$\beta =  -\frac{2p_3 \cos^2(\theta)}{p_3\cos^2(\theta) + q_3}.$$  We observe that by hypothesis, $z$ and $\beta$ are well-defined, and $\beta\neq 0$.  In particular, $Z\neq 0$.   We also note the following identities for $\beta$, which easily follow from Lemma \ref{lem:identities}:

\begin{equation}\label{eqn:beta1} \beta = \frac{2p_3\sin^2(\theta)(1-h(\theta))}{p_3\cos^2(\theta)-q_3} \end{equation} and  \begin{equation}\label{eqn:beta2} \cos^2(\theta)(1+\beta) = \sin^2(\theta)(1-h(\theta)).\end{equation}  We now proceed to verify each of the three conditions of Proposition \ref{prop:eqns}.

\bigskip

\textbf{\ref{eqn:xtrace}}

Given our choice of $A$ and $k$ above, \ref{eqn:xtrace} specializes to $$\sin^2(\theta)|\sin(\alpha) k_{11} - \cos(\alpha) \overline{k}_{21}|^2 p_3 = |k_{11}|^2 q_1 + |k_{21}|^2 q_2.$$  Using the fact that $|k_{21}|^2 = 1-|k_{11}|^2$ and equations \eqref{eqn:geqn} and \eqref{eqn:heqn}, verifying this equation is equivalent to verifying  $$h(\theta) = \frac{g(\theta)q_1 + (1-g(\theta))q_2}{\sin^2(\theta)p_3},$$ which is routine.

\bigskip

\textbf{\ref{eqn:ytrace}}

Given our choice of $A$ and $z$ above, \ref{eqn:ytrace} specializes to the equation $$((2\cos^2(\theta)-\sin^2(\theta))  p_3 + (q_1 + q_2 - 2q_3))\beta-6\sin^2(\theta) p_3|\cos(\alpha) k_{11} + \sin(\alpha) \overline{k}_{12}|^2 = 0.$$

Using equation \eqref{eqn:heqn}, we note that \begin{align*} |\cos(\alpha) k_{11} + \sin(\alpha) \overline{k}_{12}|^2 + h(\theta) &= \left| \begin{bmatrix} \cos(\alpha) k_{11} + \sin(\alpha) \overline{k}_{12} \\ -\sin(\alpha)k_{11} + \cos(\alpha)\overline{k}_{12} \end{bmatrix}\right|^2\\ &= \left| \begin{bmatrix} \cos(\alpha) & \sin(\alpha)\\ -\sin(\alpha) & \cos(\alpha)\end{bmatrix} \begin{bmatrix} k_{11}\\ \overline{k}_{21}\end{bmatrix}\right|^2\\ &= |k_{11}|^2 + |k_{21}|^2\\ &=1.\end{align*} In particular, verifying \ref{eqn:ytrace} reduces to verifying \begin{equation}\label{eqn:beta3} \beta = \frac{6\sin^2(\theta)p_3(1-h(\theta))}{(2\cos^2(\theta)-\sin^2(\theta))p_3 + (q_1 + q_2 - 2q_3)}.\end{equation}   Writing $\sin^2(\theta) = 1-\cos^2(\theta)$ and recalling that $p_3 = q_1 + q_2 + q_3$, we find that the denominator is simplifies to $3(p_3\cos^2(\theta) - q_3)$.  Then we see that equation \eqref{eqn:beta3} is nothing but equation \eqref{eqn:beta1}.
\bigskip

\textbf{\ref{eqn:dep}}

We will now verify \ref{eqn:dep}.  In fact, we will show that with our choice of $A,k,\beta$, and $z$, that $(Ad_{Ak} Y_1)_\mathfrak{p} = (Ad_{Ak} Z)_\mathfrak{p}$.  We recall that matrices in $\mathfrak{p}$ are determined entirely by their $(1,3)$ and $(2,3)$ entries.  As such, we will only list these entries.

A routine calculation shows that for generic $k = (k_{ij})\in K$ with $k_{33} = 1$ (so $k_{22} = \overline{k}_{11}$ and $k_{21} = -\overline{k}_{12}$), and generic $Z\in \mathcal{Z}$, that $(Ad_{Ak}Y_1)_\mathfrak{p}$ is given by $$3 i \sin(\theta)\begin{bmatrix} \cos^2(\alpha) k_{11} k_{12} - \sin^2(\alpha)\overline{k}_{11} \overline{k}_{12} + \sin(\alpha)\cos(\alpha)(1-2|k_{11}|^2)\\ \cos(\theta)( \cos^2(\alpha)(1-2|k_{11}|^2)-|k_{11}|^2 + \sin(\alpha)\cos(\alpha)(k_{11}k_{12} + \overline{k}_{11} \overline{k}_{12}))\end{bmatrix}.$$  Using the fact that $k\in K$, so that $|k_{11}|^2 + |k_{12}|^2 = 1$, it is easy to verify that this can be rewritten in the form $$(Ad_{Ak} Y_1)_\mathfrak{p} = 3i\sin(\theta) \begin{bmatrix} (\cos(\alpha) k_{12} - \sin(\alpha)\overline{k}_{11})(\cos(\alpha) k_{11} + \sin(\alpha) \overline{k}_{12})\\ -\cos(\theta)|\cos(\alpha) k_{12} -\sin(\alpha)\overline{k}_{11}|^2\end{bmatrix}.$$

Similarly, a routine calculation shows that for our specific choice of $z$, that $(Ad_{Ak} Z)_\mathfrak{p} = $ $$ \begin{bmatrix} 3i\sin(\theta)(\cos(\alpha)k_{12} - \sin(\alpha)\overline{k}_{11})(\cos(\alpha) k_{11} + \sin(\alpha)\overline{k}_{12})\\ 3i\tan(\theta)(\cos^2(\theta)-\sin^2(\theta))|\cos(\alpha) k_{11} + \sin(\alpha) \overline{k}_{12}|^2 + 3i\cos(\theta)\sin(\theta)\beta\end{bmatrix}.$$

It is now obvious that the $(1,3)$ entries of the equation $(Ad_{Ak} Y_1)_\mathfrak{p}) = (Ad_{Ak} Z)_{\mathfrak{p}}$ are identical.  Since $|\cos(\alpha) k_{12} - \sin(\alpha) \overline{k}_{11}|^2 = h(\theta)$ and $|\cos(\alpha) k_{11} + \sin(\alpha) \overline{k}_{12}|^2 = 1-h(\theta)$, the $(2,3)$ entries if $(Ad_{Ak} Y_1)_{\mathfrak{p}}$ and $(Ad_{Ak} Z)_{\mathfrak{p}}$ agree if and only if $$\beta = \frac{-\sin(\theta)\cos(\theta) h(\theta) - \tan(\theta)(\cos^2(\theta)-\sin^2(\theta))(1-h(\theta))}{\cos(\theta)\sin(\theta)}.$$   But this is simply a rearrangement of equation \eqref{eqn:beta2}, so it must hold.

\end{proof}

In order to use Lemma \ref{lem:important}, we need to find solutions to both \eqref{eqn:geqn} and \eqref{eqn:heqn}.  To that end, when $0<g(\theta)<1$, we will restrict attention to $k = (k_{ij})\in K$ with $k_{11} = \sqrt{g(\theta)}$, so that \eqref{eqn:geqn} is automatically satisfied.  As $|k_{11}|^2 + |k_{12}|^2 = 1$, we find that $k_{12}$ must have the form $k_{12} = \sqrt{1-g(\theta)} e^{i\gamma}$ for some $\gamma\in \mathbb{R}$.  We will find solutions to \eqref{eqn:heqn} by varying $\gamma$ and appealing to the Intermediate Value Theoerem.  This will require estimates on the left hand side of \eqref{eqn:heqn}.





\begin{lemma}\label{lem:boundlhs} Suppose that for some $\theta \in \mathbb{R}$, that $0 < g(\theta) < 1$.  Suppose $k_{11} = \sqrt{g(\theta)}\in \mathbb{R}$ and $k_{12} = \sqrt{1-g(\theta)} e^{i\gamma}$ for some $\gamma\in \mathbb{R}$.  Then we have $$\min_{\alpha\in \mathbb{R}}| \sin(\alpha) k_{11} - \cos(\alpha) \overline{k}_{12}|^2  \leq \sin^2(\gamma)$$ and $$ \cos^2(\gamma)\leq  \max_{\alpha\in \mathbb{R}}| \sin(\alpha) k_{11} - \cos(\alpha) \overline{k}_{12}|^2.$$ 

\end{lemma}

\begin{proof} By our choice of $k_{11}$ and $k_{12}$, we see $|\sin(\alpha)k_{11} - \cos(\alpha) \overline{k}_{12}|^2$ takes the form \begin{equation}\label{eqn:beta}\left(\sin(\alpha)\sqrt{g(\theta)} - \cos(\alpha)\cos(\gamma)\sqrt{1-g(\theta)}\right)^2 + \cos^2(\alpha)\sin^2(\gamma)(1-g(\theta)). \end{equation}

By selecting $\alpha$ for which $\tan(\alpha) = \frac{\cos(\gamma) \sqrt{1-g(\theta)}}{\sqrt{g(\theta)}}$, the first term in parenthesis vanishes, so $|\sin(\alpha) k_{11} - \cos(\alpha) \overline{k}_{12}|^2$ reduces to $\cos^2(\alpha)\sin^2(\gamma)(1-g(\theta))$.  Since $\cos^2(\alpha)$ and $1-g(\theta)$ are both bounded above by $1$, it follows that $\min_{\alpha\in \mathbb{R}}| \sin(\alpha) k_{11} - \cos(\alpha) \overline{k}_{12}|^2\leq \sin^2(\gamma).$

For the upper bound, we select an $\alpha$ for which $\sin(\alpha) = \sqrt{g(\theta)}$ and $\cos(\alpha) = -\sqrt{1-g(\theta)}$.  Then, by discarding the last term of \eqref{eqn:beta} we find \begin{align*} | \sin(\alpha) k_{11} - \cos(\alpha) \overline{k}_{12}|^2 &\geq \left(g(\theta)+(1-g(\theta))\cos(\gamma)\right)^2\\ &= (g(\theta)(1-\cos(\gamma)) + \cos(\gamma))^2\\ &\geq \cos^2(\gamma),\end{align*} where the last inequality holds because $g(\theta)(1-\cos(\gamma))\geq 0.$
\end{proof}

We can now prove Theorem \ref{thm:recog}.

\begin{proof}
By continuity, there is a neighborhood $U\subseteq \mathbb{R}$ of $\theta_0$ for which both $0<h(\theta_0) < 1$ and $0 < g(\theta_0) < 1$ on $U$.  Further, by shrinking $U$ if necessary, we may assume that all $\theta\in U$ satisfy all the inequalities in the statement of Lemma \ref{lem:important}.  In particular, we may assume $\theta_0$ satisfies all these inequalities.  Now, consider the subset $V\subseteq \mathcal{F}$ defined by \begin{align*}V = \big\{A(\theta,\alpha)\in\mathcal{F}: &\theta \in U \text{ and }(|\sin(\alpha)|\sqrt{g(\theta)} - |\cos(\alpha)|\sqrt{1-g(\theta)})^2 < h(\theta) \\ \text{ and } h(\theta) &< (|\sin(\alpha)|\sqrt{g(\theta)} + |\cos(\alpha)|\sqrt{1-g(\theta)})^2\big\}\end{align*}

The set $V$ is obviously an open subset of $\mathcal{F}$.  Below, we will demonstrate the following three claims:

\begin{enumerate} \item $V$ is non-empty.

\item  Under the $U(2)\times T^2$-action of Proposition \ref{prop:F}, the orbit of $V$ is an open subset of $SU(3)$.

\item  Every point of $V$ projects to a point in $SU(3)$ having at least one zero-curvature plane.

\end{enumerate}

Temporarily assuming these claims, we may complete the proof as follows.  The claims establish the existence of a non-empty open subset of $SU(3)$ consisting of points whose projections to $E_{p,q}$ all have at least one zero-curvature plane.  Since  the natural projection $SU(3)\rightarrow E_{p,q}$ is a submersion, it is open, and hence the projection of this open subset of $SU(3)$ is an open subset of $E_{p,q}$ witnessing the fact that the Wilking metric is not almost positively curved.

We now establish the claims.

\textbf{Claim 1}:  To see $V$ is non-empty, we begin with the point $\theta_0\in U$.  Select a real number  $\gamma_0>0$ with the property that $$0 < \sin^2(\gamma_0) < h(\theta_0) < \cos^2(\gamma_0) < 1.$$  Such a $\gamma_0$ exists because $0<h(\theta_0) < 1$ and both $\sin$ and $\cos$ are continuous.  Note that $\gamma_0 < \pi/4$ for otherwise the inequality $\sin^2(\gamma_0) < \cos^2(\gamma_0)$ is false. 
 We let $k_{11} = \sqrt{g(\theta)}\in \mathbb{R}$ and $k_{12} = \sqrt{1-g(\theta)} e^{i\gamma_0}.$

Now, by Lemma \ref{lem:boundlhs}, the real function $\alpha \mapsto |\sin(\alpha) k_{11} - \cos(\alpha) \overline{k}_{12}|^2$ has a range containing the closed interval $[\sin^2(\beta_0),\cos^2(\beta_0)]$.  In particular, it contains $h(\theta_0)$.  Thus, by the Intermediate Value Theorem, there is an $\alpha_0\in\mathbb{R}$ for which $$|\sin(\alpha_0)k_{11} -\cos(\alpha_0) \overline{k}_{12}|^2 = h(\theta_0).$$

We now show that $A(\theta_0,\alpha_0)\in V$, so that $V\neq \emptyset.$  To see this,  simply note that the triangle inequality implies that both $$\left||\sin(\alpha_0)|\sqrt{g(\theta_0)} - |\cos(\alpha)|\sqrt{1-g(\theta_0)}\right|\leq |\sin(\alpha_0) k_{11} -\cos(\alpha_0) \overline{k}_{12}|$$ and that $$|\sin(\alpha_0) k_{11} -\cos(\alpha_0) \overline{k}_{12}| \leq  |\sin(\alpha_0) | \sqrt{g(\theta_0)} + |\cos(\alpha)| \sqrt{1-g(\theta_0)}. $$  Moreover, equality can only occur when $\gamma_0$ is an integral multiple of $\pi$.  Recalling that $\gamma_0\in (0,\pi/4)$, we see that both inequalities are strict.  Squaring both strict inequalities and using the fact that $|\sin(\alpha_0)k_{11} - \cos(\alpha_0) \overline{k}_{12}|^2 = h(\theta_0)$, it now follows that $A(\theta_0,\alpha_0)\in V$ so $V\neq \emptyset.$

\textbf{Claim 2}:   We next claim that, under the $U(2)\times T^2$ action of Proposition \ref{prop:F}, that the orbit of $V$ is an open subset of $SU(3)$.  Indeed, following Corollary \ref{cor:sameorbit}, $A = (a_{ij}) \in SU(3)$ is orbit equivalent to an element of $\mathcal{F}$ where $|a_{33}|^2 = \cos^2(\theta), |a_{31}|^2 = \sin^2(\theta) \sin^2(\alpha)$, and $|a_{32}|^2 = \sin^2(\theta)\cos^2(\alpha)$.  These latter two equalities can be rearranged to $|\sin(\alpha)| = \frac{|a_{31}|}{\sqrt{1-|a_{33}|^2}}$ and $|\cos(\alpha)| = \frac{|a_{32}|}{\sqrt{1-|a_{33}|^2}}.$  We also note that by replacing all $\sin^2(\theta)$ with $1-\cos^2(\theta) = 1-|a_{33}|^2$, we may view both $h$ and $g$ as functions of $a_{33}$.

Then Corollary \ref{cor:sameorbit} implies that the $U(2)\times T^2$ orbit of $V$ is given by \begin{align*} &\left\{ A = (a_{ij})\in SU(3): \left(\frac{|a_{31}|}{\sqrt{1-|a_{33}|^2}} \sqrt{g(a_{33})} - \frac{|a_{32}|}{\sqrt{1-|a_{33}|^2}} \sqrt{1-g(a_{33})}\right)^2 < h(a_{33})\right. \\ &\text{and } h(a_{33}) < \left. 
 \left(\frac{|a_{31}|}{\sqrt{1-|a_{33}|^2}} \sqrt{g(a_{33})} + \frac{|a_{32}|}{\sqrt{1-|a_{33}|^2}} \sqrt{1-g(a_{33})}\right)^2\right\},  \end{align*} which is obviously open.

\bigskip

\textbf{Claim 3}:  We now show that Lemma \ref{lem:important} implies that every point in $V$ has at least one zero-curvature plane.  So, let $A(\theta,\alpha)\in V$.  We set $k_{11} = \sqrt{g(\theta)}$, so that \eqref{eqn:geqn} is satisfied.  Writing $k_{12} = \sqrt{1-g(\theta)} e^{i\gamma}$, our goal is to find $\gamma$ solving \eqref{eqn:heqn}.

We observe that by the reverse triangle inequality, we have the bounds $$\left| |\sin(\alpha)| \sqrt{g(\theta)} - |\cos(\alpha)|\sqrt{1-g(\theta)}\right| \leq |\sin(\alpha) k_{11} - \cos(\alpha) \overline{k}_{12}|$$ with equality achieved at either $\gamma  = 0$ or $\gamma = \pi$, depending on the signs of $\sin(\alpha)$ and $\cos(\alpha)$.  Similarly, by the triangle inequality we have the bound $$|\sin(\alpha) k_{11} - \cos(\alpha) \overline{k}_{12}| \leq |\sin(\alpha)|\sqrt{g(\theta)} + |\cos(\alpha)|\sqrt{1-g(\theta)}$$ which equality similarly achieved at either $\gamma = 0$ or $\gamma = \pi$.

Squaring these inequalities, the fact that $A(\theta,\alpha)\in V$ implies that  $h(\theta)$ lies between the two extremes of the function $\gamma \mapsto  |\sin(\alpha) k_{11} - \cos(\alpha)\overline{k}_{12}|^2$.  Thus, the Intermediate Value Theorem implies the existence of a $\gamma_0\in \mathbb{R}$ for which $|\sin(\alpha) k_{11} -\cos(\alpha) \overline{k}_{12}|^2 = h(\theta)$.  That is, \eqref{eqn:heqn} is solved by this choice of $\gamma_0$.  Lemma \ref{lem:important} now implies the existence of a zero-curvature plane at $[A(\theta,\alpha)]$.

\end{proof}

We now use Theorem \ref{thm:recog} to show that, under the hypotheses of Proposition \ref{prop:choices}, that Wilking's metric is not almost positively curved.  That is, we will now prove Theorem \ref{thm:almpos}.

\begin{proof}Assume $q_1,q_2,$ and $q_3$ fulfill conditions $1a)$ through $1d)$ of Proposition \ref{prop:choices}.  That is, we assume that $p_3 > 0$, $q_1 >0$, $q_1 > q_2$, and that $q_2 q_3 < 0$ so $q_2$ and $q_3$ have opposite signs.  We will break the proof into cases based on the cases $2a)$, $2b)$, and $2c)$ in Proposition \ref{prop:choices}.

\textbf{Case 2c)}:  Assume that $q_3 < 0$ and $q_1 + q_3 > 0$.  We will additionally assume that $q_2 + q_3\leq 0$; the opposite case where $q_2  + q_3>0$ will be accounted for in Case 2a) below.  We also observe that the following proof does not use the hypothesis that $q_3 < 0$.

 We compute that $h(\pi/2) = 1$, $h'(\pi/2) = 0$, and $h''(\pi/2) = -2$, so that $0 < h(\theta) < 1$ for all $\theta$ sufficiently close to $\pi/2$.

 We also compute that $g(\pi/2) = \frac{q_1 + q_3}{q_1-q_2}$, $g'(\pi/2) = 0$, and $g''(\pi/2) = \frac{-4p_3}{q_1-q_2}< 0$.  By hypothesis, the numerator and denominator of $g(\pi/2)$ are both positive, so $g(\pi/2) > 0$.

 In addition, we also see that $g(\pi/2) \leq 1$.  Indeed, we have \begin{align*} q_2 + q_3 &\leq  0\\
 q_3 &\leq -q_2\\
 q_1 + q_3 &\leq q_1 - q_2\\
 \frac{q_1+q_3}{q_1-q_2} &\leq 1.\end{align*}

 It follows that for all $\theta$ sufficiently close to $\pi/2$, so $0 < g(\theta) < 1$.  Since we have shown $0< h(\theta)<1$ for $\theta$ sufficiently close to $\pi/2$, we may apply Theorem \ref{thm:recog} to complete the proof of this case.

\bigskip

\textbf{Case 2a)}:  Here, we assume that $q_2 + q_3 \geq 0$.  Since $q_1 > q_2$ by hypothesis, $q_1 + q_3 > 0$.  In particular, if $q_2 + q_3 = 0$, it follows under Case 2c) above.  As such, for the remainder of the proof of this case, we may assume that $q_2 + q_3 > 0$.

We define $\theta_0\in(0,\pi/2)$ by equation $$\cos^2(\theta_0) = \frac{-q_3(q_2+q_3)}{(q_2-q_3)p_3}.$$  To see this is well-defined, we need to verify that $0< \frac{-q_3(q_2+q_3)}{(q_2-q_3)p_3}< 1$.  The case where $q_2 <0$ and $q_3 > 0$ is similar to the case where $q_2 > 0$ and $q_3 < 0$, so we only show the first case.  So, assume that $q_2 < 0$ and $q_3 > 0$.  Then the numerator and denominator are both negative, so the fraction is positive.  In addition, since $q_1 > 0$ and $p_3 > 0$, we see that both $q_1 q_3 > 0$ and $-q_2p_3>$, so that $0 < q_1 q_3 - q_3 p_3$.  Then \begin{align*} 0 &< q_1 q_3 - q_2 p_3\\
q_3(q_2+q_3)&< q_3(q_2+q_3) + q_1 q_3 - q_2p_3\\
q_3(q_2+q_3)&< q_3p_3 - q_2p_3\\
 q_3(q_2 + q_3)&< (q_3-q_2)p_3\\
 \frac{q_3(q_2+q_3)}{(q_3-q_2)p_3}&< 1\\
 \frac{-q_3(q_2+q_3)}{(q_2-q_3)p_3}&< 1.\end{align*}

 A simple computation now reveals that $g(\theta_0) = 1$.  Further, from Lemma \ref{lem:identities}, $g'(\theta_0) \neq 0$.  In particular, any neighborhood of $\theta_0$ contains points $\theta$ where $0<g(\theta)< 1$.  We also compute that $h(\theta_0) = \frac{(q_2 - q_3)q_1}{q_2p_3 - q_1 q_3}$.  We claim that $0 < h(\theta_0) < 1$.  Believing this, it follows that any neighborhood of $\theta_0$ contains points $\theta$ where both $0< h(\theta) < 1$ and $0<g(\theta)<1$.  Then Theorem \ref{thm:recog} completes the proof in this case.

 It remains to establish the claim.  Again, the case where $q_2 < 0$ and $q_3 > 0$ is similar to the case where $q_2 > 0$ and $q_3<0$, so we only prove it in the first case.  So, assume that $q_2 < 0$ and $q_3 > 0$.  The numerator of $h(\theta_0)$ is negative, as is the denominator, so $h(\theta_0) >0$.  On the other hand, since $q_2 + q_3 > 0$, we have \begin{align*} q_1 &< p_3\\
 q_1 q_2 &>  q_2p_3\\
 q_1 q_2 - q_1 q_3 &> q_2p_3 - q_1 q_3\\
 (q_2-q_3) q_1 &> q_2 p_3 - q_1 q_3\\
 \frac{(q_2-q_3)q_1}{q_2p_3 - q_1 q_3} &< 1.\end{align*}

 The completes the proof in case 2a).

 \textbf{Case 2b)}:  Assume that $q_2 < 0$ and that $q_1 + q_2 \geq 0$.  We break into subcases depending on whether $q_1 + q_2 =0$ or not.  So, assume initially that $q_1 + q_2 = 0$.  Then on $(0,\pi/2)$,  $h(\theta)$ simplifies to $h(\theta) = \frac{1}{\cos^2(\theta)+1}$ while $g(\theta) = \frac{q_1(1+\cos^2(\theta)) + q_3\sin^2(\theta)}{2q_1(1+\cos^2(\theta))}$.  Then $0 < h(\theta) < 1$ for all $\theta \in (0,\pi/2)$ and $g(0) = \frac{1}{2}.$  Thus, for all $\theta > 0$ sufficiently close to $0$, we have $0 < h(\theta) < 1$ and $0 < g(\theta) < 1$.  So, by Theorem \ref{thm:recog}, there is an open set of points having a zero-curvature plane in this case as well.

 We may thus assume that $q_1 + q_2 > 0$.  Then $0 < \frac{q_3}{p_3} < 1$, so there is a $\theta_0\in \mathbb{R}$ with the property that $\cos^2(\theta_0) = \frac{q_3}{p_3}$.  We then compute $h(\theta_0) = 1$ and, from Lemma \ref{lem:identities}, that $h'(\theta_0) \neq 0$.  In particular, any neighborhood of $\theta_0$ contains points for which $0<h(\theta)<1$.  Moreover, $g(\theta_0) = \frac{q_1}{q_1-q_2}$.  By our hypotheses on the $q_i$, it follows that $0<g(\theta_0) < 1$.  By continuity, for all $\theta$ sufficiently close to $\theta_0$, $0<g(\theta) < 1$.  It follows that any neighborhood of $\theta_0$ contains a point for which both $0 < h(\theta) <1$ and $0<g(\theta)<1$.  One final application of  Theorem \ref{thm:recog} then completes the proof in this case, and hence completes the proof of Theorem \ref{thm:almpos}.

\end{proof}

\end{document}